\documentclass[reqno,12pt]{amsart}

\usepackage[usenames]{color}
\usepackage{enumerate}
\usepackage{amssymb}

\usepackage{amsmath,amsthm,amsfonts,amssymb,enumerate,amscd}

\usepackage{graphicx}
\usepackage{graphicx}
\usepackage{epsfig}
\usepackage{colortbl}
\usepackage{tikz}
\usepackage{multicol}
\usepackage{verbatim}
\usetikzlibrary{arrows}

\newtheorem{teorema}{Theorem}[section] 
\newtheorem{theorem}{Theorem}[section]

\newtheorem{corollary}[teorema]{Corollary}

\newtheorem{proposition}[teorema]{Proposition}

\newtheorem{definition}[teorema]{Definition}

\newtheorem{question}{Question}

\newtheorem{example}[teorema]{Example}

\newtheorem{remark}[teorema]{Remark}

\def\RR{\mathbb R}

\author[]{Teresa Berm\'udez}
\email{tbermude@ull.es}

\author[]{Antonio Bonilla}
\email{abonilla@ull.es}

\author[]{N. S. Feldman}
\email{feldmanN@wlu.edu}

\address{Departamento de An\'{a}lisis Matem\'{a}tico, 
  Universidad de La Laguna, 
  38271 La Laguna (Tenerife), Spain}

\address{Departamento de An\'{a}lisis Matem\'{a}tico, 
  Universidad de La Laguna, 
  38271 La Laguna (Tenerife), Spain}

\address{Dept. of Mathematics, Washington and Lee University, Lexington VA 24450}

\date{\today}

\title{On convex-cyclic operators}


\begin{document}

\maketitle

\begin{abstract}
We give a Hahn-Banach Characterization for convex-cyclicity.  We also obtain an example of  a bounded linear operator $S$ on a Banach space with $\sigma_{p}(S^*)=\emptyset$ such that $S$ is convex-cyclic, but $S$ is not weakly hypercyclic and  $S^2 $ is
not convex-cyclic. This solved two questions of Rezaei in \cite{Rezaei} when $\sigma_p(S^*)=\varnothing$.
 We also characterize the diagonalizable normal operators that are convex-cyclic and give a condition on the eigenvalues of an arbitrary  operator for it to be convex-cyclic.  We  show that certain adjoint multiplication operators are convex-cyclic and show that some are convex-cyclic but no convex polynomial of the operator is hypercyclic.  Also some adjoint multiplication operators are convex-cyclic but not 1-weakly hypercyclic.

\end{abstract}

\section{Introduction}

Let $X$ be a Banach space and let $L(X)$ denote the algebra  of  all bounded  linear  operators on $X$.  A bounded linear operator  $T$ on $X$ is \emph{cyclic} if  there exists a (cyclic) vector $x$ such that the linear span  of the orbit of $x$, $Orb(T,x)= \{ T^n x: n=0,1, \cdots \}$, is dense in $X$.   An operator $T$ is called \emph{convex-cyclic}  if there exists a vector  $x\in X$ such that the convex  hull of $Orb(T,x)$ is dense in $X$ and such a vector $x$ is said to be a \emph{convex-cyclic vector} for $T$. Clearly all convex-cyclic operators   are cyclic.  Following Rezaei~\cite{Rezaei}  we will say that a polynomial $p$ is a \textit{convex polynomial} if it is a (finite) convex combination of monomials $\{1, z, z^2, \ldots \}$.  So, $p(z) = a_0 + a_1 z + \cdots + a_nz^n$ is a convex polynomial if $a_k \geq 0$ for all $k$ and $\sum_{k=0}^n a_k = 1$.  Then the convex hull of an orbit is $co(Orb(T,x)) = \{ p(T)x : p \text{ is a convex polynomial} \}$.

A bounded linear operator  $T\in L(X)$ is said to be \emph{hypercyclic} (\emph{weakly hypercyclic} \cite{ChanSanders}) if there is a vector $x\in X$ whose orbit is dense in the norm (weak) topology of $X$.   An operator $T$ is said to be \emph{weakly-mixing} if $T\oplus T$ is hypercyclic in $X\oplus X$.

 There are certainly examples of convex-cyclic operators that are not hypercyclic. However within  certain classes of operators, hypercyclicity and convex-cyclicity are equivalent.   This is true for unilateral weighted backward shifts on $\ell ^p(\mathbb{N})$ and composition operators on the classical Hardy space, see \cite{Rezaei}.

 What follows is a list of questions that are answered in this paper.  First, notice that
every weakly hypercyclic operator  is convex-cyclic since the norm and the weak closure of a convex set in a Banach space coincide.  In \cite{Rezaei} Rezaei asks the following question:

\begin{question}\cite[Question 5.4]{Rezaei} \label{weakly-hypercyclic}
Is every convex-cyclic operator acting on an infinite dimensional  Banach space weakly hypercyclic?
\end{question}

According to Feldman \cite{Feldman}, $T$ is called \emph{1-weakly hypercyclic} if there is an $x\in X$ such that $f(Orb(T,x))$ is dense in $\Bbb C$ for each non-zero $f\in X^*$. Every weakly hypercyclic operator is 1-weakly hypercyclic and $1$-weakly hypercyclic operators are convex-cyclic. Thus it is also natural to ask if every convex-cyclic operator acting on an infinite dimensional Banach space 1-weakly hypercyclic?

\ \par

 Ansari \cite{Ansari}  showed that  powers of hypercyclic  operators on Banach spaces are hypercyclic  operators.   The same result was proven for operators on locally convex spaces by Bourdon and  Feldman~\cite{BoFe}. These results do not have analogues for cyclic operators. The forward unilateral shift $S$  on $\ell^2(\mathbb{N} )$ is cyclic but $S^2$ is not cyclic, because the codimension of the range of $S^2$ is two.  What about powers of convex-cyclic operators? Le\'on and Romero in \cite{LeRo}  give  examples of convex-cyclic operators where $\sigma_p(S^*)$ is non-empty that have powers that are not convex-cyclic. Thus is natural to ask:

 \begin{question}\label{potencias}\cite[Question 5.5]{Rezaei}
 If $S: X\rightarrow X$ a convex-cyclic operator on a Banach space $X$ with $\sigma_{p}(S^*)= \emptyset$, then is $S^{n}$   convex-cyclic for every integer $n>1$?
\end{question}

For a positive integer $m$ and a positive real number $p$, an operator $T\in L(X)$ is called an \emph{$(m,p)$-isometry} if for any $x\in X$,

$$
\sum _{k=0}^m (-1)^{m-k}{ m\choose k}\|T^{k}x\|^p =0\; .
$$

An operator $T$ is called an \emph{$m$-isometry} if it is an $(m,p)$-isometry for some $p> 0$. See \cite{agler}, \cite{Bayart} and \cite{hms}.  Faghih and Hedayatian  proved in \cite{FaHe} that $m$-isometries on a Hilbert space are not weakly hypercyclic.  However, there are isometries that are weakly supercyclic \cite{Sanders05} (in particular cyclic). Thus a natural question is the following:

 \begin{question}\label{m-isometrias}
Can an  $m$-isometry be convex-cyclic?
\end{question}

In \cite{BaGrMu}, Badea,  Grivaux and  M\"uller introduced the concept of an $\varepsilon$-hypercyclic operator.
\begin{definition}
{\rm Let $\varepsilon \in (0,1)$ and let $T: X\rightarrow X$ be a continuous linear operator. A vector $x\in X$ is called   an \emph{$\varepsilon$-hypercyclic vector} for $T$ if for every non-zero vector $y\in X$ there exists a non-negative integer $n$ such that
$$
\|T^nx-y\|\le \varepsilon \|y\|.
$$
The operator $T$ is called \emph{$\varepsilon$-hypercyclic} if it has an $\varepsilon$-hypercyclic vector.}
\end{definition}

 In \cite{BaGrMu} it was shown that for every $\varepsilon \in (0,1)$, there exists an $\varepsilon$-hypercyclic operator on the space $\ell ^1(\Bbb N)$ which is not hypercyclic. Bayart in \cite{Bayart-1} extended this result to separable Hilbert spaces. Thus it is natural to ask if:

\begin{question}\label{varepsilon-hypercyclic}
Is every $\varepsilon$-hypercyclic operator  also convex-cyclic?
\end{question}

An operator $T\in L(X)$ is called \emph{hypercyclic with support $N$} is there exists a vector $x\in X$ such that  the set
$$
\left\{  T^{k_1}x+ T^{k_2} x+ \cdots + T^{k_N} x\;\;\; : \;\;\; k_1,\; \ldots , \; k_N\in \mathbb{N} \right\}
$$
is dense in $X$.

\begin{remark}{\rm
Notice that if $T$ is hypercyclic with support $N$, then $T$ is convex-cyclic. In fact,
 for any $y\in X$, there exist $k_1,\ldots , k_N\in \mathbb{N} $ such that $T^{k_1}x+ T^{k_2} x+ \cdots + T^{k_N} x \approx Ny$, thus
 $$
 \frac{T^{k_1}x+ T^{k_2} x+ \cdots + T^{k_N} x}{N} \approx y \;.
 $$}
\end{remark}

Any hypercyclic operator   with support $N$ satisfies that $\sigma_p (T^*)$ is the empty set \cite[Proposition 3.1]{Bayart-Costakis}. However, there are convex-cyclic operators such that $\sigma_p (T^*)$ is non-empty. So,  hypercyclicity with support $N$ is not equivalent to convex-cyclicity.

In  \cite{Rezaei}, Rezaei characterizes which diagonal matrices on $\mathbb{C}^n$ are convex-cyclic as those whose eigenvalues are distinct and belong to the set $\mathbb{C} \setminus (\overline{\mathbb{D}} \cup \mathbb{R})$.  This naturally leads to the question about infinite diagonal matrices and even the following more general question.

\begin{question} \label{complete set of eigenvectors}
If $T$ is a continuous linear operator on a complex Banach space $X$ and $T$ has a complete set of eigenvectors whose eigenvalues are distinct, and belong to the set $\mathbb{C} \setminus (\overline{\mathbb{D}} \cup \mathbb{R})$, then is $T$ convex-cyclic?
\end{question}

In this paper, we answer these five questions and also give some examples.  The paper is organized as follows. In Section 2, we give the Hahn-Banach characterization  for convex-cyclicity. In Section 3, we  give an example of an operator $S$ that is convex-cyclic but $S^2$  is not convex-cyclic and thus $S$ is not weakly hypercyclic, this answers Questions \ref{weakly-hypercyclic} and \ref{potencias} when $\sigma_p(S^*)=\varnothing$. In Section 4, we prove that  $m$-isometries are not convex-cyclic, answering Question \ref{m-isometrias}. In Section 5 we prove that any $\varepsilon$-hypercyclic  operator is convex-cyclic. In fact, every $\varepsilon$-hypercyclic vector is a convex-cyclic vector.  Finally, in Section 6 we answer Question 5 affirmatively and give examples of such operators including diagonal operators and adjoints of multiplication operators.

\ \par

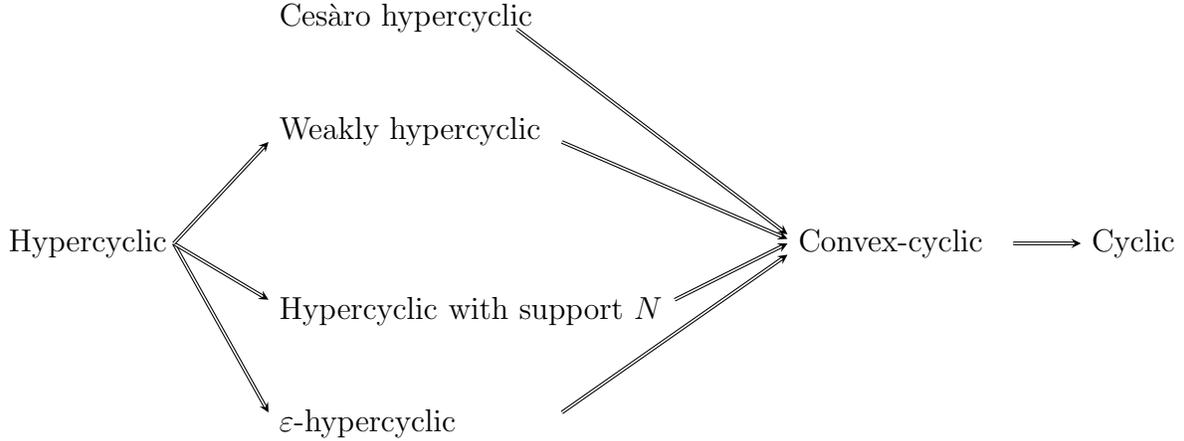
\begin{figure}[h]
\hspace*{-1cm}
\begin{tikzpicture}[scale=0.3,>=stealth]
 \node[right] at (-5,0) {Hypercyclic};
  \node[right] at (7,10) {Ces\`{a}ro hypercyclic};
 \node[right] at (7,5) {Weakly hypercyclic};
  \node[right] at (7,-3) {Hypercyclic with support $N$};
  \node[right] at (7,-8) {$\varepsilon$-hypercyclic};
  \node[right] at (30,0) {Convex-cyclic};
 \node[right] at (43,0) {Cyclic};
 \draw[double, ->] (2.8,0) -- (7,-7.5);
  \draw[double, ->] (2.8,0) -- (7,-2.5);
  \draw[double, ->] (2.8,0) -- (7,4.5);
\draw[double, ->] (20,-7.5) -- (30,-.5);
  \draw[double, ->] (25,-2.5) -- (30,0);
  \draw[double, ->] (20,4.5) -- (30,0.2);
    \draw[double, ->] (18,9.5) -- (30,0.4);
\draw[double, ->] (40,0) -- (43,0);
 \end{tikzpicture}
\caption{Implications between different definitions related with hypercyclicity and cyclicity.}
\end{figure}

\section{The Hahn-Banach Characterization for Convex-Cyclicity}

Rezaei gave a (universality) criterion for an operator to be convex-cyclic \cite[Theorem 3.10]{Rezaei}. In the following result, using the Hahn-Banach Separation Theorem, we give a necessary and sufficient condition for a set to have a dense convex hull, as a result we get a criterion for a vector to be a convex-cyclic vector for an operator.

\begin{proposition}  \label{P:linearfunctionals}
Let $X$ be a locally convex space over the real or complex numbers and let $E$ be a nonempty subset of $X$.  The following are equivalent:
\begin{enumerate}
\item The convex hull of $E$ is dense in $X$.
\item For every nonzero continuous linear functional $f$ on $X$ we have that the convex hull of $Re(f(E))$ is dense in $\mathbb{R}$.
\item For every nonzero continuous linear functional $f$ on $X$ we have that $$\sup Re(f(E)) = \infty \text{ and } \inf Re(f(E)) = -\infty.$$
\item For every nonzero continuous linear functional $f$ on $X$ we have that $$\sup Re(f(E)) = \infty.$$
\end{enumerate}
\end{proposition}

\begin{proof}
Let $\mathbb{F}$ denote either the real or complex numbers.  Clearly $(1) \Rightarrow (2) \Rightarrow (3) \Rightarrow (4)$ holds.  Now assume that \textit{(4)} holds and by way of contradiction, assume that $co(E)$ is not dense in $X$.    Then there exists a point $p \in X$ that is not in the closure of $co(E)$.  So, by the Hahn-Banach Separation Theorem  (\cite[Theorem 3.13]{Conway}), there exists a continuous linear functional $f$ on $X$ so that $Re(f(x)) < Re(f(p))$ for all $x \in co(E)$.  It follows that $Re(f(E))$ is bounded from above and thus $\sup Re(f(E)) \neq \infty$.  This contradicts our assumption that \textit{(4)} is true.  Thus it must be the case that if \textit{(4)} holds, then \textit{(1)} does also.  Hence all four conditions are equivalent.
\end{proof}

\begin{corollary}[The Hahn-Banach Characterization for Convex-Cyclicity]  \label{C:linearfunctionalcriteria}
Let $X$ be a locally convex space over the real or complex numbers, $T:X \to X$ a continuous linear operator, and $x \in X$.  Then the following are equivalent:
\begin{enumerate}
\item The convex hull of the orbit of $x$ under $T$ is dense in $X$.
\item For every non-zero continuous linear functional $f$ on $X$ we have $$\sup Re(f(Orb(T,x)) )= \infty.$$
\end{enumerate}
\end{corollary}

Below are some simple consequences of the Hahn-Banach characterization for convex-cyclic vectors.

As it was pointed in the Introduction the range of a cyclic operator may  not be dense. For example, the range of the unilateral shift has codimension one.  However, the closure of the range of a cyclic operator has codimension at most one. Notice that the range of hypercyclic operator  is always dense.   The Hahn-Banach characterization of convex-cyclicity easily shows that  convex-cyclic operators must also have dense range, see the following result.

\begin{proposition}\label{2.3}
If $T$ is a convex-cyclic operator on a locally convex space $X$, then $T$ has dense range.
\end{proposition}

\begin{proof}
Suppose that $T$ is a convex-cyclic operator and let $x$ be a convex-cyclic vector for $T$, and by way of contradiction, suppose that $T$ does not have dense range.  Then there exists a continuous linear functional $f$ such that $f(R(T)) = \{0\}$, where $R(T)$ denotes the range of $T$.    By the Hahn-Banach characterization, Corollary~\ref{C:linearfunctionalcriteria}, we must have that $\sup Re(f(Orb(T,x)) = \infty$.  However, since $T^nx \in R(T)$ for all $n \geq 1$ it follows that $f(T^nx) = 0$ for all $n \geq 1$.  So, $\sup Re(f(Orb(T,x)) = \sup Re(\{f(T^0x), 0\})  < \infty$.  It follows from Corollary~\ref{C:linearfunctionalcriteria} that $x$ is not a convex-cyclic vector, a contradiction.  Thus, $T$ must have dense range.
\end{proof}

In general, if $T$ is hypercyclic and $c>1$, then $cT$ may not be hypercyclic. However, Le\'{o}n-Saavedra and M\"{u}ller \cite{MuLe} proved that if $T$ is hypercyclic and $\alpha $ is a unimodular complex number, then $\alpha T$ is hypercyclic. The same property is also true for  weak hypercyclic operators \cite[Theorem 2.8]{Rosa}. Next we present a similar result for convex-cyclic operators, that follows from the Hahn-Banach characterization of convex-cyclic vectors.

\begin{proposition}\label{cT}
If $T$ is a convex-cyclic operator on a real or complex  locally convex space $X$, and if $c > 1$, then $cT$ is also convex-cyclic.  Furthermore, every convex-cyclic vector for $T$ is also a convex-cyclic vector for $cT$.
\end{proposition}

\begin{proof}
Suppose that $x$ is a convex-cyclic vector for $T$, and we will show that $x$ is also a convex-cyclic vector for $cT$, by using the Hahn-Banach characterization (Corollary~\ref{C:linearfunctionalcriteria}).  Let $f$ be any non-zero continuous linear functional on $X$.  Since $x$ is a convex-cyclic vector for $T$, then $\sup Re(f(T^nx)) = \infty.$  Since $c > 1$, then we have that $\sup Re[f((cT)^nx)] = \sup c^n Re[f(T^nx)] \geq  \sup Re[f(T^nx)] = \infty$.  So, by the Hahn-Banach characterization, $x$ is a convex-cyclic vector for $cT$.
\end{proof}

\begin{corollary}
If $|c|\geq 1$ and $T$ is weakly hypercyclic, then $cT$ is convex-cyclic.
\end{corollary}

\begin{proof}
Let $c:= e^{i\theta } \beta $, where $\theta \in \RR$ and $\beta \geq 1$. Then by  de la Rosa \cite[Theorem 2.8]{Rosa} we obtain that  $e^{i\theta }T $ is weakly hypercyclic, hence $e^{i\theta }T $ is convex-cyclic.  Thus, $cT = \beta (e^{i\theta}T)$ is convex cyclic by Proposition \ref{cT}.
\end{proof}

 Let us define the following convex polynomials
$$
p^c_k(t):= \left\{
\begin{array}{ll}
\displaystyle \frac{1+t+\cdots +t^{k-1}}{k} & \mbox{ if } c=1\\[1pc]
\displaystyle\frac{c-1}{c^k-1} (c^{k-1}+c^{k-2}t+ \cdots + t^{k-1})& \mbox{ if } c>1 \;.
\end{array}
\right.
$$

\begin{definition}
{\rm Let $X$ and $Y$ be topological spaces.  A family of continuous operators $T_i:X \rightarrow Y$ $(i\in I)$ is  \emph{universal} if there exists an $x\in X$ such that $\{ T_ix: i\in I\}$ is dense in $Y$.
}
\end{definition}

Let  $T\in L(X)$. Denotes $M_n(T)$ the \emph{arithmetic means} given by
$$
M_n(T):= \frac{I+T+ \cdots +T^{n-1}}{n}\;.
$$
Recall that an operator $T$ is  \emph{Ces\`{a}ro hypercyclic} if there exists $x\in X$ such that $\{ M_n(T)x\;\; : \;\; n\in \mathbb{N}\}$ is dense in $X$. See \cite{fernandoleon}.

In \cite[Theorem 2.4]{fernandoleon} it is proved that $T$ is Ces\'{a}ro hypercyclic if and only if $\left(\displaystyle \frac{T^k}{k}\right)_{k=1}^\infty$ is universal.

\begin{proposition}\label{cesaro-convex}
Let $X$ be a Banach space, $c >1$ and $T\in L(X)$ such that $cI-T$ has  dense range. Then the following are equivalent:
\begin{enumerate}
\item  $\displaystyle \frac{T}{c}$ is hypercyclic
  \item $\left( p_k^c(T)\right)_{k\in \mathbb{N}} $ is universal.
\end{enumerate}
\end{proposition}
\begin{proof}
Notice that if $c>1$,
$$
p_k^c(T)(cI-T)x=(cI-T)p_k^c(T)x=
(c-1)\frac{c^k}{c^k-1} \left(x-\left(\frac{T}{c}\right)^kx\right) \; .
$$
\end{proof}

\begin{proposition} \label{cesaro-hypercyclic}
If $T$ is Ces\`{a}ro hypercyclic or $\displaystyle \frac{T}{c}$  is hypercyclic for some $c\geq 1$, then $T$ is convex-cyclic.
\end{proposition}

Notice that the proof of the sufficient condition for a bilateral weighted backward shift on $\ell^p(\mathbb{Z})$ to be convex-cyclic given in  \cite[Theorem 4.2]{Rezaei} is not  correct.

\section{Convex-cyclic operators whose squares are not convex-cyclic}

 As noted in the Introduction, powers of hypercyclic and weakly hypercyclic  operators remain hypercyclic and weakly hypercyclic, respectively. In this section, we give an example  of a convex-cyclic operator  $S$ with $\sigma_{p}(S^*)=\emptyset$ such that $S^2$ is not convex-cyclic. Moreover, the same example gives an operator that is convex-cyclic with $\sigma_{p}(S^*)=\emptyset$  that is not weakly hypercyclic.

Recall that Le\'{o}n-Saavedra and Romero de la Rosa \cite{LeRo} provide an example of a convex-cyclic operator $S$ with $\sigma_p(S^*)\neq \varnothing$  such that $S^n$ fails to be convex-cyclic.  Also, a $2 \times 2$  diagonal matrix $D$ with eigenvalues $2i$ and $-2i$ is convex-cyclic, but $D^2$ has a real eigenvalue and thus is not convex-cyclic.

\begin{theorem}\label{theo1}\cite{Grivaux} Let $T$ be a hypercyclic operator on an infinite dimensional  separable Banach space.
The following assertions are equivalent:
\begin{enumerate}
\item $T\oplus T$ is hypercyclic.

\item $T\oplus T$ is cyclic.

\end{enumerate}

\end{theorem}

\begin{theorem}(\cite[Proposition 2.3]{Bayart-Costakis} \& \cite[Corollary 5.2]{shkarin2}) Let $T$ be a hypercyclic operator on a separable Banach space. Then $T\oplus -T$ is hypercyclic with support 2 and 1-weakly hypercyclic.
\end{theorem}

\begin{corollary}\label{3.1}
If $T$ is a hypercyclic operator  on an infinite dimensional  Banach space such that $T\oplus T$ is not hypercyclic, then $T\oplus -T$ is  convex-cyclic, but not weakly hypercyclic  and $(T\oplus -T)^2$ is not cyclic.
\end{corollary}

\begin{proof}
Suppose that $T$ is a hypercyclic operator such that $T\oplus T$ is not hypercyclic.  Then by Theorem \ref{theo1}, $T\oplus T$ is not cyclic. Thus $(T\oplus T)^2$ is not cyclic, hence   $(T\oplus -T)^2=(T\oplus T)^2$ is  not cyclic. It follows that  $T\oplus -T$ is not weakly hypercyclic, for if it was, then $(T \oplus -T)^2 = (T \oplus T)^2$ would be weakly hypercyclic, and hence cyclic, a contradiction.  Thus $T \oplus -T$ is convex-cyclic but not weakly hypercyclic, and $(T \oplus -T)^2$ is not cyclic.
\end{proof}

Examples of operators satisfying that $T$ is hypercyclic but $T \oplus T$ is not hypercyclic are given in \cite{Rosa-Read}, \cite[Corollary 4.15]{Bayart-Matheron2} and \cite{Bayart-Matheron1}.  Using these examples we have the following result.

\begin{theorem} There exists an  operator $S$ on $c_{0}(\Bbb N)\oplus c_{0}(\Bbb N)$ or on
$\ell^p(\Bbb N)\oplus \ell^p(\Bbb N)$ with $p\geq 1$ that is  convex-cyclic,  but not weakly hypercyclic, and $S^2$ is not convex-cyclic.
\end{theorem}

Using similar ideas of   Shkarin \cite[Lemma 6.5]{shkarin1} we obtain the following result.

\begin{theorem}
Let  $T \in L(X)$. If $T^2$ is convex-cyclic, then $T\oplus -T$ is convex-cyclic.
\end{theorem}

\begin{proof}
Let $x$ be a convex-cyclic vector  for $T^2$ and let $S:=T\oplus -T$. Then for all $y\in X$ there exists a sequence $(p_k)$ of convex polynomials such that $p_k(T^2)x$ converges to $y$ as $k$ tends to infinity.
Thus
$$
p_k(S^2) (x,x) \to (y,y)\; ,
$$
and
$$
Sp_k(S^2)(x,x)\to (Ty,-Ty) \;.
$$
Since $T^2$ is convex-cyclic, $T$ is convex-cyclic. By Proposition \ref{2.3}, the range of $T$ is dense.
By other hand, $p_k(x^2) $ and $xp_k(x^2)$ are convex polynomials. Thus the closed convex hull of $Orb(S,(x,x))$ contains
 the spaces $L_0:=\{ (u,u) : u\in X\}$ and $L_1:=\{ (u,-u) : u\in X\}$.
  So, if we are given $(y,z)\in X\times X$, then
let $(q_k)$ and $(h_k)$ be sequences of convex polynomials such that
$$
q_k(S^2)(x,x) \rightarrow (y+z,y+z)
$$
and
$$
Sh_k(S^2)(x,x) \rightarrow (y-z,z-y) \; .
$$
Then $p_k(t):=\frac12 q_k(t^2)+\frac{t}{2}h_k(t^2)$ is a sequence of convex polynomials  and
$$
p_k(S)(x,x) \rightarrow(y,z) \; .
$$
Thus $S = T \oplus -T$ is convex-cyclic.
\end{proof}


Ansari \cite{Ansari} proved that an operator $T$ is hypercyclic if and only if $T^n$ is hypercyclic. In fact $T$ and $T^n$ have the same set of hypercyclic vectors for any positive integer $n$.
 This property is also true for weakly hypercyclic  vectors (see \cite[Theorem 2.4]{BoFe}), thus we get the following corollary.

\begin{corollary}
If $T$ is weakly-hypercyclic, then $T\oplus -T$ is convex-cyclic.
\end{corollary}

In the following result we obtain that if  $T$ and $T^n$ are convex-cyclic operators, the set of convex-cyclic vectors could be different.

\begin{proposition}
There are hypercyclic operators such that $T$ and $T^2$  do not  have the same convex-cyclic vectors.
\end{proposition}
\begin{proof}
Let $T$ be twice the backward shift, $T:=2B$, on $\ell ^2(\mathbb{N})$ and let $D$ be the doubling map on $\ell^2(\mathbb{N})$, given by $D(x_0, x_1, x_2, \ldots) = (x_0, x_0, x_1, x_1, x_2, x_2, \ldots)$.
By \cite[Theorem 5.3]{Feldman} there exists an $x \in \ell^2(\mathbb{N})$ such that $x$ is a  1-weakly hypercyclic  vector for $T$ (and hence a convex-cyclic vector for $T$) and  $Orb(T^2, x) \subseteq D(\ell^2(\mathbb{N}))$.  Thus,
$$
\overline{co (Orb(T^2, x))}\subseteq \overline{span[ Orb(T^2,x) ] } \subseteq  D(\ell ^2 (\mathbb{N})) \neq \ell^2(\mathbb{N}) \; .
$$
 Since $D(\ell^2(\mathbb{N}))$ is  a proper  closed subspace of $\ell^2(\mathbb{N})$, this complete the proof.
\end{proof}

 \section{m-isometries are not convex-cyclic}

Bayart proved the following spectral result for $m$-isometries on Banach spaces.

\begin{proposition}\label{propo} \cite[Proposition 2.3]{Bayart} Let $T\in L(X)$ be an m-isometry. Then its approximate point spectrum lies in the unit circle. In particular, $T$ is one-to-one, $T$ has closed range and either $\sigma(T) \subseteq \Bbb T$ or $\sigma (T) = \overline {\Bbb D}$.
 \end{proposition}

 On the other hand, Rezaei proved the following properties for convex-cyclic operators.

 \begin {proposition} \label{propo2} \cite [Propositions 3.2 and 3.3] {Rezaei} Let $T\in L(X)$. If $T$ is convex-cyclic, then

 \begin{enumerate}
 \item $\|T\|>1$.

 \item $\sigma _{p}(T^*)\subset \Bbb C\setminus (\overline {\Bbb D} \cup \Bbb R)$.

 \end{enumerate}

 \end{proposition}

%
%
%

\begin{theorem} An $m$-isometry on a  Banach space $X$ is not  convex-cyclic.
\end{theorem}

\begin{proof}

If $m=1$ and $T$ is an $m$-isometry, then $T$ is actually an isometry, thus $\|T\| = 1$ and thus by part (1) of  Proposition \ref{propo2}, $T$ cannot be convex-cyclic.

Assume that  $m\geq 2$ and that $T$ is convex-cyclic and  a strict $(m,p)$-isometry for some $p>0$. We will use an argument similar to the proof of \cite[Theorem 3.3]{Bayart}.
Let
$$
\displaystyle |x|: = \lim_{n\to \infty } \frac{\|T^{n}x\|}{n^{\frac{m-1}{p}}}\;.
 $$
 By \cite[Proposition 2.2]{Bayart} we have that $|.|$ is a semi-norm on $X$ and   $T(Ker(|.|) \subset Ker(|.|)$, where $Ker (T)$ denotes the  kernel of $T$.
Also the codimension of $Ker(| .|) $ is positive, because $T$ is not a $(m-1)$-isometry.
Moreover, for each $x\in X$, $|Tx| = |x|$ and there exists $C>0$ such that $|x|\leq C\| x\| $ for all $x\in X$.

Let $Y: = X/Ker(|.|)$ and $ \overline{T}$ be the operator induced by $T$ on $Y$. Then  $|\overline{T} \overline{x}|= |\overline{x}|$ for all $\overline{x} \in Y$. So,
$\overline T$ is an isometry on $Y$.

 Since $T$ is convex-cyclic there exists a vector  $x\in X$ such that the convex  hull generated by $Orb(T,x)$ is dense in $X$.
  Given $y\in X$ and $\varepsilon >0$ there exists a convex polynomial such that $\displaystyle \|y- p_{n}(T)x\|< \frac{\varepsilon}{C}$.
  Thus $|y- p_{n}(T)x|\leq C\| y-p_n(T)x\|<\varepsilon$. Then $|\overline y- p_{n}(\overline T)\overline x|< \varepsilon$ and we obtain that $\overline T$ is convex-cyclic in $Y$.

Thus the extension of  $\overline T$ to the completion of $Y$ is a convex-cyclic isometry on a Banach space; which  is a contradiction.
\end{proof}

\begin{corollary}
An $m$-isometry on a Banach space is not 1-weakly hypercyclic.
\end{corollary}


\section{ $\varepsilon$-hypercyclic operators versus  convex-cyclic operators}

Let us now exhibit the relation between $\varepsilon$-hypercyclic and convex-cyclic operators.

\begin{theorem} Every $\varepsilon$-hypercyclic vector  is a convex-cyclic vector.
 \end{theorem}
\begin{proof}
Let $x$ be an  $\varepsilon$-hypercyclic vector for an operator $T$ and we will prove that for  a non-zero vector $y\in X$ and $\delta >0$, there exists a convex polynomial $p$ such that
$$
\|p(T)x-y\|<\delta \;.
$$
Since $\varepsilon \in (0,1)$, there exists  $N\in \Bbb N$ such that $2\varepsilon^N\|y\|<\delta$. As $x$ is an  $\varepsilon$-hypercyclic  vector for $T$, there exists a positive integer $k_1$ such that
$$
\left\|T^{k_1}x-Ny\|\le \varepsilon \|Ny\right\|=\varepsilon N\|y\| \;.
$$
If $T^{k_1}x-Ny=0$, we choose $l_2$ such that
$$
\left\|T^{l_2}x-\frac{N}{N-1}\varepsilon ^{N}y\right\|\le \varepsilon ^{N+1} \frac{N}{N-1}\|y\| \;.
$$
Thus
$$
\left\|\frac{N-1}{N}T^{l_2}x-\varepsilon ^Ny\right\|\le \varepsilon ^{N+1}\|y\| \;.
$$
Hence
$$
\left\|\frac{1}{N}T^{k_1}x+\frac{N-1}{N}T^{l_2}x-y\right\| = \left\| \frac{N-1}{N} T^{l_2} x\right\| \le 2\varepsilon ^N\|y\|<\delta
$$
and the proof ends by letting $p(z) = \frac{1}{N}z^{k_1} + \frac{N-1}{N}z^{l_2}$.

If  $T^{k_1}x-Ny \neq0$, there exists a positive integer $k_2$ such that
$$
\left\|T^{k_1}x+T^{k_2}x-Ny\|=\|T^{k_2}x-(Ny-T^{k_1}x)\right\|\le\varepsilon \|Ny-T^{k_1}x\|\le\varepsilon^2 N\|y\| \;.
$$

If $T^{k_1}x+T^{k_2}x-Ny=0$, analogously to the above situation we choose $l_3$ such that
$$
\left\|\frac{1}{N}T^{k_1}x+\frac{1}{N}T^{k_2}x+\frac{N-2}{N}T^{l_3}x-y\right\| = \left\| \frac{N-2}{N} T^{l_3} x \right\|\le 2\varepsilon ^N\|y\|<\delta
$$
and the proof ends.

If  $T^{k_1}x+T^{k_2}x-Ny\neq 0$, there exists a positive integer $k_3$ such that
$$
\left\|T^{k_1}x+T^{k_2}x+T^{k_3}x-Ny\right\|\le\varepsilon^3 N\|y\| \;.
$$

By induction, in the step $N$,
if $T^{k_1}x+T^{k_2}x+ \cdots +T^{k_{N-1}}x-Ny=0$, we choose $l_N$ such that
$$
\left\|\frac{1}{N}T^{k_1}x+\frac{1}{N}T^{k_2}x+\cdots+\frac{1}{N}T^{k_{N-1}}x+ \frac{1}{N}T^{l_N}x-y\right\|\le 2\varepsilon ^N\|y\|<\delta
$$
and the proof ends.

If  $T^{k_1}x+T^{k_2}x+ \cdots +T^{k_{N-1}}x-Ny\neq0$, there exists a positive integer $k_N$ such that
$$
\left\|T^{k_1}x+T^{k_2}x+\cdots+T^{k_{N-1}}x+ T^{k_N}x-Ny\right\|\le \varepsilon ^N N\|y\|
$$

Thus
$$
\left\|\frac{T^{k_1}x+\cdots +T^{k_N}x}{N}-y\right\|\le \varepsilon^N \|y\|<\delta
$$
Ending completely the proof.
\end{proof}

\section{ Diagonal Operators and  Adjoint Multiplication Operators}

By a Fr\'{e}chet space we mean a locally convex space that is complete with respect to a translation invariant metric.

If $\mathcal{A}$ is a nonempty collection of polynomials and $T$ is an operator on a space $X$, then $T$ is said to be \emph{$\mathcal{A}$-cyclic} and $x \in X$ is said to be an \emph{$\mathcal{A}$-cyclic vector} for $T$ if $\{ p(T)x : p \in \mathcal{A}\}$ is dense in $X$.  Furthermore, $T$ is said to be \emph{$\mathcal{A}$-transitive} if for any two nonempty open sets $U$ and $V$ in $X$, there exists a $p \in \mathcal{A}$ such that $p(T)U \cap V \neq \emptyset$.  Since the set of all polynomials with the topology of uniform convergence on compact sets in the complex plane forms a separable metric space, then any set of polynomials is also separable, hence the following result is routine (see for example the Universality Criterion in \cite[Theorem 1.57]{Grosse-Peris}).

\begin{proposition}  \label{P:A-transitive}
Suppose that $T:X \to X$ is a continuous linear operator on a real or complex Fr\'{e}chet space and $\mathcal{A}$ is a nonempty set of polynomials.  Then the following are equivalent:
\begin{enumerate}
\item $T$ has a dense set of $\mathcal{A}$-cyclic vectors.
\item $T$ is $\mathcal{A}$-transitive.  That is, for any two nonempty open sets $U, V$ in $X$, there is a polynomial $p \in \mathcal{A}$ such that $p(T)U \cap V \neq \emptyset$.
\item $T$ has a dense $G_\delta$ set of $\mathcal{A}$-cyclic vectors.
\end{enumerate}
\end{proposition}

By choosing various sets of polynomials for $\mathcal{A}$, we can get results for hypercyclic and supercyclic operators, as well as cyclic operators that have a dense set of cyclic vectors.  If $\mathcal{A}$ is the set of all convex polynomials, then we get the following immediate corollary.

\begin{corollary}  \label{C:convex transitive}
Let $T:X \to X$ be a continuous linear operator on a real or complex Fr\'{e}chet space, then the following are equivalent.
\begin{enumerate}
\item $T$ has a dense set of convex-cyclic vectors.
\item $T$ is convex-transitive.  That is, for any two nonempty open sets $U, V$ in $X$, there is a convex polynomial $p$ such that $p(T)U \cap V \neq \emptyset$.
\item $T$ has a dense $G_\delta$ set of convex-cyclic vectors.
\end{enumerate}
\end{corollary}

\begin{proposition}  \label{P:infinitedirectsums}
Let  $\mathcal{A}$ be a nonempty set of polynomials and let $\{T_k:X_k \to X_k\}_{k=1}^\infty$ be a uniformly bounded sequence of linear operators on a sequence of Banach spaces $\{X_k\}_{k=1}^\infty$ such that for every $n \geq 1$, the operator $\displaystyle S_n = \bigoplus_{k=1}^n T_k$ on $\displaystyle X^{(n)} = \bigoplus_{k=1}^n X_k$ has a dense set of $\mathcal{A}$-cyclic vectors.  Then $\displaystyle T = \bigoplus_{k=1}^\infty T_k$ is $\mathcal{A}$-cyclic on $\displaystyle X^{(\infty)} = \bigoplus_{k=1}^\infty X_k$ and $T$ has a dense set of $\mathcal{A}$-cyclic vectors.
\end{proposition}

\begin{proof}
Suppose that for every $n \geq 1$ the operators $S_n$ are $\mathcal{A}$-cyclic and have a dense set of $\mathcal{A}$-cyclic vectors.  We will show that $T$ is $\mathcal{A}$-transitive.  Let $U$ and $V$ be two nonempty open sets in $X^{(\infty)}$.  Since the vectors in $X$ with only finitely many non-zero coordinates are dense in $X$, then we may choose vectors $x = (x_k)_{k=1}^\infty$ and $y = (y_k)_{k=1}^\infty$  in $X^{(\infty)}$ such that $x_k = 0$ and $y_k = 0$ for all sufficiently large $k$, say $x_k =0$ and $y_k = 0$ for all $k \geq N$, and such that $x \in U$ and $y \in V$.   Since $S_N$ is $\mathcal{A}$-cyclic and has a dense set of $\mathcal{A}$-cyclic vectors in $X^{(N)}$, there exists a vector $u = (u_1, u_2, \ldots, u_N) \in X^{(N)}$ such that $u$ is an $\mathcal{A}$-cyclic vector for $S_N$ and so that  $(u_1, u_2, \ldots, u_N)$ is close enough to $(x_1, x_2, \ldots, x_N)$ so that  the infinite vector $\hat{u} = (u_1, u_2, \ldots, u_N, 0, 0, \ldots) \in U$.  Since $S_N$ is $\mathcal{A}$-cyclic, there is a polynomial $p \in \mathcal{A}$ such that $p(S_N)(u_1, u_2, \ldots, u_N)$ is close enough to $(y_1, y_2, \ldots, y_N)$ such that $p(T) \hat{u} \in V$.  Thus, $T$ is $\mathcal{A}$-transitive on $X^{(\infty)}$, and thus by Proposition~\ref{P:A-transitive} we have that  $T$ has a dense set of $\mathcal{A}$-cyclic vectors.
\end{proof}

We next apply the previous proposition to infinite diagonal  operators where $\mathcal{A}$ is the set of all convex polynomials.   This extends the finite dimensional matrix result given by Rezaei~\cite[Corollary 2.7]{Rezaei} to infinite dimensional diagonal matrices.

\begin{theorem} \label{T:diagonalnormaloperators} Suppose that $T$ is a diagonalizable normal operator on a separable (real or complex) Hilbert space with eigenvalues $\{\lambda_k\}_{k=1}^\infty$.

(a) If the Hilbert space is complex, then $T$ is convex-cyclic if and only if  we have that the eigenvalues $\{\lambda_k\}_{k=1}^\infty$ are distinct and for every $k \geq 1$, $| \lambda_k | > 1$  and $Im(\lambda_k) \neq 0$.	

(b) If the Hilbert space is real, then $T$ is convex-cyclic if and only if the eigenvalues $\{\lambda_k\}_{k=1}^\infty$ are distinct and for every $k \geq 1$ we have that $\lambda_k < -1$.	
\end{theorem}

\begin{proof}  By the spectral theorem we may assume that $T = diag(\lambda_1, \lambda_2, \ldots)$ is an infinite diagonal matrix acting on  $\ell^2_{\mathbb{C}}(\mathbb{N})$ and let $\{e_k\}_{k=1}^\infty$ be the canonical unit vector basis where $e_k$ has a one in its $k^{th}$ coordinate and zeros elsewhere.

(a)    If $T$ is convex-cyclic with convex-cyclic vector $x = (x_n)_{n=1}^\infty \in \ell^2_{\mathbb{C}}(\mathbb{N})$,  then by Corollary~\ref{C:linearfunctionalcriteria} we must have for every $k \geq 1$ that $\infty = \sup_{n \geq 1} Re(\langle T^nx, e_k \rangle ) = \sup_{n \geq 1} Re(\lambda_k^n x_k)$.  This implies that $x_k \neq 0$ and that $|\lambda_k| > 1$  for each $k \geq 1$.  Likewise, since the Hilbert space is complex in this case, we must have
$$
\infty = \sup_{n \geq 1} Re\left(\langle T^nx, \frac{-i}{\overline{x_k}} e_k \rangle \right) = \sup_{n \geq 1} Re\left (\lambda_k^n x_k \frac{i}{x_k} \right) = \sup_{n \geq 1} Re(i \lambda_k^n)\;.
 $$
This implies that $\lambda_k$ cannot be real, hence $Im(\lambda_k) \neq 0$ for all $k \geq 1$.

Conversely, suppose that  for every $k \geq 1$ we have that $| \lambda_k | > 1$  and $Im(\lambda_k) \neq 0$.    Then for $n \geq 1$, let $T_n := diag(\lambda_1, \lambda_2, \ldots, \lambda_n)$ be the diagonal matrix on $\mathbb{C}^n$ where $\lambda_k$ is the $k^{th}$ diagonal entry.  Since the eigenvalues $\{\lambda_k\}_{k=1}^\infty$ are distinct and $| \lambda_k | > 1$  and $Im(\lambda_k) \neq 0$ for $1 \leq k \leq n$, then we know from Rezaei~\cite{Rezaei} that $T_n$ is convex-cyclic on $\mathbb{C}^n$ and that every vector all of whose coordinates are non-zero is a convex-cyclic vector for $T_n$.  Since such vectors are dense in $\mathbb{C}^n$  for every $n \geq 1$, then it follows from Proposition~\ref{P:infinitedirectsums} that $T$ is also convex-cyclic and has a dense set of convex-cyclic vectors.  (b) The proof of the real case is similar to that above.
\end{proof}

The next theorem says that if an operator has a complete set of eigenvectors whose eigenvalues are distinct, not real, and lie outside of the closed unit disk, then the operator is convex-cyclic.

\begin{theorem}  \label{T:dense set of eigenvectors}
Let $S := \{re^{i\theta} : r > 1 \text{ and } 0 < |\theta| < \pi \} = \mathbb{C} \setminus (\overline{\mathbb{D}} \cup \mathbb{R})$.  Suppose that $T$ is a bounded linear operator on a complex Banach space $X$ and that $T$  has a countable linearly independent set of eigenvectors with dense linear span in $X$ such that the corresponding eigenvalues are distinct and are contained in the set $S$.  Then $T$ is convex-cyclic and has a dense set of convex-cyclic vectors.
\end{theorem}

\begin{proof}
Suppose that $\{v_n\}_{n=1}^\infty$ is a linearly independent set of eigenvectors for $T$ that have dense linear span in $X$ and such that the corresponding eigenvalues $\{\lambda_n\}_{n=1}^\infty$ are distinct and contained in the set $S$.  By replacing each eigenvector $v_n$ with a constant multiple of itself we may assume that $\sum_{n=1}^\infty \|v_n\|^2 < \infty$.  Let $D$ be the diagonal normal matrix on $\ell^2(\mathbb{N})$ whose $n^{th}$ diagonal entry is $\lambda_n$.  Then define a linear map $A:\ell^2(\mathbb{N}) \to X$ by $A(\{a_n\}_{n=1}^\infty) = \sum_{n=1}^\infty a_n v_n$.  Notice that since $\{a_n\}_{n=1}^\infty \in \ell^2(\mathbb{N})$, then we have that
$$\|A(\{a_n\}_{n=1}^\infty)\| =   \left \|\sum_{n=1}^\infty a_n v_n \right \| \leq \left( \sum_{n=1}^\infty |a_n|^2 \right)^{1/2} \left ( \sum_{n=1}^\infty \|v_n\|^2 \right)^{1/2} = C \|\{a_n\}_{n=1}^{\infty} \|_{\ell^2(\mathbb{N})}$$
where $C := ( \sum_{n=1}^\infty \|v_n\|^2 )^{1/2}$, which is finite.  The above inequality implies that $A$ is a well defined continuous linear map from $\ell^2(\mathbb{N})$ to $X$.  It follows that since the eigenvectors $\{v_n\}_{n=1}^\infty$ have dense linear span in $X$, that $A$ has dense range.    Also, if $\{e_n\}_{n=1}^\infty$ is the standard unit vector basis in $\ell^2(\mathbb{N})$, then clearly $A(e_n) = v_n$ for all $n \geq 1$ and thus
$A$ intertwines $D$ with $T$.  To see this notice that  $AD(e_n) = A(\lambda_n e_n) = \lambda_n v_n = T(v_n) = TA(e_n)$.  Thus $AD(e_n) = TA(e_n)$ for all $n \geq 1$, thus $AD = TA$.
Finally, since $D$ has distinct eigenvalues that all lie in the set $S$, it follows from Proposition~\ref{T:diagonalnormaloperators}  that $D$ is convex-cyclic and has a dense set of convex-cyclic vectors.  Since $A$ intertwines $D$ and $T$ and $A$ has dense range, then $A$ will map convex-cyclic vectors for $D$ to convex-cyclic vectors for $T$.  Thus, $T$ is convex-cyclic and has a dense set of convex-cyclic vectors.
\end{proof}

If $G$ is an open set in the complex plane, then by a reproducing kernel Hilbert space $\mathcal{H}$ of analytic functions on $G$ we mean a vector space of analytic functions on $G$ that is complete with respect to a norm given by an inner product and such that point evaluations at all points in $G$ are continuous linear functionals on $\mathcal{H}$.  Naturally we also require that $f = 0$ in $\mathcal{H}$ if and only if $f(z) = 0$ for all $z \in G$.   This is equivalent to the reproducing kernels having dense linear span in $\mathcal{H}$.   Given such a space $\mathcal{H}$, a multiplier of $\mathcal{H}$ is an analytic function $\varphi$ on $G$ so that $\varphi f \in \mathcal{H}$ for every $f \in \mathcal{H}$.  In this case, the closed graph theorem implies that the multiplication operator $M_\varphi:\mathcal{H} \to \mathcal{H}$ is a bounded linear operator.

\begin{corollary}  \label{C:adjoint multipliers}
Suppose that $G$ is an open set in $\mathbb{C}$ with components $\{G_n\}_{n \in J}$ and $\mathcal{H}$ is a reproducing kernel Hilbert space of analytic functions on $G$, and that $\varphi$ is a multiplier of $\mathcal{H}$.  If $\varphi$ is non-constant on every component of $G$ and $\varphi(G_n)  \cap \{z \in \mathbb{C} : |z| > 1\} \neq \emptyset$ for every $n \in J$, then the operator $M^*_\varphi$ is convex-cyclic on $\mathcal{H}$ and has a dense set of convex-cyclic vectors.
\end{corollary}

\begin{proof}
We will show that the eigenvectors for $M_\varphi^*$ with eigenvalues in the set $S = \mathbb{C} \setminus (\overline{\mathbb{D}} \cup \mathbb{R})$ have dense linear span in $\mathcal{H}$.  It will then follow from Theorem~\ref{T:dense set of eigenvectors} that $M_{\varphi}^*$ is convex-cyclic.

Every reproducing kernel for $\mathcal{H}$ is an eigenvector for $M_\varphi^*$.  In fact, if $\lambda \in G$, then $M^*_\varphi K_\lambda = \overline{\varphi(\lambda)} K_\lambda$, where $K_\lambda$ denotes the reproducing kernel for $\mathcal{H}$ at the point $\lambda \in G$.  By assumption, for every component $G_n$ of $G$, $\varphi$ is non-constant on $G_n$, thus the set $\{\lambda \in G_n : |\varphi(\lambda)| > 1\}$ is a nonempty open subset of $G_n$. Also since $\varphi$ is an open map on $G_n$,  $\varphi$ cannot map the open set $\{\lambda \in G_n : |\varphi(\lambda)| > 1\}$ into $\mathbb{R}$.  Thus, for all $n \in J$, $E_n = \{\lambda \in G_n : |\varphi(\lambda)| > 1 \text{ and } \varphi(\lambda) \notin \mathbb{R}\}$ is a nonempty open subset of $G_n$.  Let $E: = \bigcup_{n \in J} E_n$.  Then for every $\lambda \in E$, $K_\lambda$ is an eigenvector for $M_\varphi^*$ with eigenvalue $\overline{\varphi(\lambda)}$ which lies in $S = \mathbb{C} \setminus (\overline{\mathbb{D}} \cup \mathbb{R})$.  Since $E \cap G_n$ is a nonempty open set for every $n \in J$, then the corresponding reproducing kernels $\{K_\lambda : \lambda \in E\}$ have dense linear span in $\mathcal{H}$.    Finally, since $\varphi$ is non-constant on $E_n$ for each $n \in J$,  we can choose a countable set $\{\lambda_{n,k}\}_{k=1}^\infty$ in $E_n$ that has an accumulation point in $E_n$ in such a way that $\varphi$ is one-to-one on $\{\lambda_{n,k}\}_{n,k=1}^\infty$.  Then the countable set $\{K_{\lambda_{n,k}}\}_{n,k=1}^\infty$ is a set of independent eigenvectors with dense linear span in $\mathcal{H}$ and with distinct eigenvalues.
It now follows from Theorem~\ref{T:dense set of eigenvectors} that $M_\varphi^*$ is convex-cyclic and has a dense set of convex-cyclic vectors.
\end{proof}

\begin{remark}
In the previous corollary, if $G$ is an open {connected} set, $\varphi$ is a non-cons\-tant\-mul\-ti\-plier of $\mathcal{H}$ and if the norm of $M_\varphi$ is {equal} to its spectral radius, then $M_\varphi^*$ is convex-cyclic {if and only if}   $\varphi(G) \cap \{z \in \mathbb{C} : |z| > 1\} \neq \emptyset$.  This is the case if $\mathcal{H}$ is equal to $H^2(G)$ or $L^2_a(G)$, the Hardy space or Bergman space on $G$ or if $M_\varphi$ is hyponormal.
\end{remark}

 Next we give an example of a convex-cyclic operator that  is not 1-weakly hypercyclic.

\begin{example}
{\rm  Let $M^*_{2+z}$ be the adjoint of the multiplication operator associated to the multiplier $\varphi (z):=2+z$ on $H^2(\mathbb{D})$.  By \cite[Theorem 5.5]{shkarin2} the operator $M^*_{2+z}=2I+B$, where $B$ is the unilateral backward shift, is not 1-weakly-hypercyclic, however $M^*_{2+z}$ is convex-cyclic by Corollary \ref{C:adjoint multipliers}.}
\end{example}

The following result is true since powers of convex polynomials are also convex polynomials.

\begin{proposition}
If $T$ is an operator on a Banach space and there exists a convex polynomial $p$ such that $p(T)$ is hypercyclic, then $T$ is convex-cyclic.
\end{proposition}

By a \textit{region} in $\mathbb{C}$ we mean an open connected set in $\mathbb{C}$.  In the following theorem, we consider the operator which is the adjoint of multiplication by $z$, the independent variable.

\begin{theorem}
Suppose that $G$ is a bounded region in $\mathbb{C}$ and $G \cap \{ z : |z| > 1\} \neq \emptyset$.  Suppose also that $\mathcal{H}$ is a reproducing kernel Hilbert space of analytic functions on $G$, then $M_z^*$ is convex-cyclic on $\mathcal{H}$. In fact, there exists a convex polynomial $p$ such that $p(M_z^*)$ is hypercyclic on $\mathcal{H}$.
\end{theorem}

\begin{proof}
Choose $n \geq 1$ such that $G^n := \{z^n : z \in G\}$ satisfies $G^n \cap \{z \in \mathbb{C} : Re(z) < 1\} \neq \emptyset$.   To see how to do this, choose a polar rectangle $R = \{r e^{i\theta} : r_1 < r < r_2 \text{ and }\alpha < \theta < \beta \}$ such that $R \subseteq G$.  Then simply choose a positive integer $n$ such that $n(\beta - \alpha) > 2\pi$.  Then $R^n \subseteq G^n$ and $R^n$ will contain the annulus $\{re^{i\theta} : r_1^n < r < r_2^n\}$, so certainly $G^n \cap \{z \in \mathbb{C} : Re(z) < 1\} \neq \emptyset$.   Now if $0 < a \leq 1$, then the convex polynomial $p_a(z) = az + (1-a)$ maps the disk $B(\frac{a-1}{a},\frac{1}{a})$ onto the unit disk.  Notice that the family of disks $\{B(\frac{a-1}{a},\frac{1}{a}) : 0 < a < 1\}$ is the family of all disks that are centered on the negative real axis and pass through the point $z = 1$.   Thus it follows that $\{z \in \mathbb{C} : Re(z) < 1 \} = \bigcup_{0 < a < 1} B(\frac{a-1}{a},\frac{1}{a})$.
So we can choose an $a \in (0,1)$ such that $G^n \cap \partial B(\frac{a-1}{a},\frac{1}{a}) \neq \emptyset$.
It follows that the polynomial $p(z) = p_a(z^n)$ is a convex polynomial and furthermore it satisfies $p(G) \cap \partial \mathbb{D} \neq \emptyset$.

Thus $M_p^*$ is hypercyclic on $\mathcal{H}$.  However, $M_p^* = p^\#(M_z^*)$ where $p^\#(z) = \overline{p(\overline{z})}$.  Also, since $p$ is a convex polynomial, all of its coefficients are real, thus $p^\# = p$.   Thus, $p(M_z^*) = p^\#(M_z^*) = M_p^*$ is hypercyclic on $\mathcal{H}$.
\end{proof}

In the next result we give an example of an operator that is convex-cyclic but no convex polynomial of the operator is hypercyclic.  In other words, the operator is purely convex-cyclic.

\begin{example}
Let $\{\alpha_n\}_{n=1}^\infty$ and $\{\beta_n\}_{n=1}^\infty$ be two strictly decreasing sequences of positive numbers that are interlaced and converging to zero.  In other words, $0 < \alpha_{n+1} < \beta_{n+1} < \alpha_n$ for all $n \geq 1$ and $\alpha_n \to 0$ (and hence $\beta_n \to 0$).
For each $n \geq 1$, let $$G_n: = \{re^{i\theta} : 2 < r < 2 + \frac{1}{n} \text{ and } \alpha_n < \theta < \beta_n\}.$$
Let $G := \bigcup_{n=1}^\infty G_n$ and let $L^2_a(G)$ be the Bergman space of all analytic functions on $G$ that are square integrable with respect to area measure on $G$.  Then the operator $M_z^*$ is purely convex-cyclic on $L^2_a(G)$; meaning that $M_z^*$ is convex-cyclic on $L^2_a(G)$, but  $p(M_z^*)$ is not hypercyclic on $L^2_a(G)$ for any convex polynomial $p$.
\end{example}

\begin{proof}
By Corollary~\ref{C:adjoint multipliers} we know that $M_z^*$ is convex-cyclic on $L^2_a(G)$.  In order to show that no convex polynomial of $M_z^*$ is hypercyclic, suppose, by way of contradiction, that there exists a convex polynomial $p$ such that $p(M_z^*)$  is hypercyclic.  Since $p$ is a convex polynomial it has real coefficients thus $p^{\#}(z) = p(z)$ where $p^{\#}(z): = \overline{p(\overline{z})}$.  Thus  $p(M_z^*) = M_{p^\#}^* = M_p^*$ and it follows that $M_p^*$ is hypercyclic on $L^2_a(G)$.  Thus it follows that every component $G_n$ of $G$ must satisfy that $p(G_n) \cap \partial \mathbb{D} \neq \emptyset$.  However since $p$ is a convex polynomial, $p$ is (strictly) increasing on the interval $[0, \infty)$.  Thus, $p(2) > p(1) = 1$.  Choose an $\varepsilon > 0$ such that $\varepsilon < p(2) - 1$.  Since $p$ is continuous at $z = 2$, and since we have an $\varepsilon > 0$, then there exists a $\delta > 0$ such that if $|z-2| < \delta$, then $|p(z) - p(2)| < \varepsilon$.  Notice that for $n$ sufficiently large we have that $G_n \subseteq B(2, \delta)$, thus, $p(G_n) \subseteq B(p(2),\varepsilon) \subseteq \{z \in \mathbb{C} : Re(z) > 1 \}$.  Thus, $p(G_n) \cap \partial \mathbb{D} = \emptyset$ for all large $n$, a contradiction.    It follows that no convex polynomial of $M_z^*$ is hypercyclic, hence $M_z^*$ is purely convex-cyclic.
\end{proof}

\section{Open Questions}

It is well known that hypercyclic operators have a dense set of hypercyclic vectors. In fact, the set of hypercyclicic vectors is a dense $G_{\delta}$ set.

\setcounter{question}{0}

\begin{question}
If $T$ is convex-cyclic, then does $T$ have a dense set of  convex-cyclic vectors?
\end{question}

Sanders \cite{Sanders05} proved that if $T:H\rightarrow H$ is  a hyponormal operator on a Hilbert space $H$, then $T$ is not weakly hypercyclic. A hyponormal operator is \emph{pure} if its restriction to any of its reducing subspaces is not normal.  That is, a hyponormal operator $T$ is pure if $T$ cannot be written in the form $T = S \oplus N$ where $N$ is a normal operator.

\begin{question}
{\rm Are there pure hyponormal operators or continuous normal operators that are convex-cyclic?}
\end{question}

\begin{question}
If $T$ is convex-cyclic on a complex Hilbert space, then is $(-1)T$ also convex-cyclic?
\end{question}

The above question is true for diagonal normal operators/matrices and the other examples in this paper and also whenever $T^2$ is convex-cyclic.

\begin{question}
{\rm If $T$ is a convex-cyclic operator, then how big can the point spectrum of $T^*$ be?  Can it have non-empty interior?
}
\end{question}

Bourdon and Feldman \cite{BoFe} showed that  if a vector $x\in X$   has a somewhere dense  orbit under a bounded linear operator $T$, then the orbit of $x$ under $T$ must be everywhere dense in $X$. A similar question was posed  for convex-cyclicity by Rezaei. Recently, Le\'{o}n-Saavedra and Romero de la Rosa provide an example where Bourdon and Feldman's result fails for convex-cyclic operators $T$ such that $\sigma_p(T^*)\neq \varnothing$.

\begin{question}\cite[Question 5.5]{Rezaei}
{\rm Let $X$ be a Banach space and $T\in L(X)$  where $\sigma_p(T^*)= \emptyset$. If $x \in X$ and $co(Orb(T,x))$ is somewhere dense in $X$, then is $co(Orb(T,x))$ dense in $X$?
}
\end{question}

Since it is unknown if there exists a Banach space on which every hypercyclic operator is weakly mixing, we ask:

\begin{question}
Given a separable Banach space $X$, is there a convex-cyclic operator $S$ on $X$ such that $S^2$ is not convex-cyclic?
\end{question}

{\bf Ackmowledgements:}
The first and second  authors are partially supported by grant of Ministerio  de Ciencia e Innovaci\'{o}n, Spain, proyect no. MTM2011-26538.

\end{document}